\documentclass[12pt]{iopart}

\usepackage{epstopdf}
\usepackage{subfigure}

\usepackage{latexsym,bm}
\usepackage{xcolor,graphicx}
\usepackage{eso-pic}
\usepackage{caption}
\usepackage{subfigure}
\newtheorem{lemma}{Lemma}[section]
\newtheorem{theorem}{Theorem}[section]

\newtheorem{proof}{Proof}
\newtheorem{remark}{Remark}[section]

\makeatletter

\@addtoreset{equation}{section}
\makeatother

\begin{document}

\title{Unified analysis on Petrov-Galerkin method into Symm's integral of the first kind}
\author{Yidong Luo}
\address{School of Mathematics and Statistics, Wuhan University, Hubei Province, P. R. China}
\ead{Sylois@whu.edu.cn}
\vspace{10pt}
\begin{indented}
\item[]August 2017
\end{indented}

\begin{abstract}
    On bounded and simply connected planar analytic domain $ \Omega $, by $ 2\pi $ periodic parametric representation of boundary curve $ \partial \Omega $, Symm's integral equation of the first kind takes form $ K \Psi = g $, where $ K $ is seen as an operator mapping from $ L^2(0,2\pi) $ to itself. The classical result show complete convergence and error analysis in $ L^2 $ setting for least squares, dual least squares, Bubnov-Galerkin methods with Fourier basis when $ g \in H^r(0,2\pi), \ r \geq 1 $.
    \newline \indent In this paper, weakening the boundary $ \partial \Omega $ from analytic to $ C^3 $ class, we maintain the convergence and error analysis from analytic case. Besides, it is proven that, when $ g \in H^r(0,2\pi), \ 0 \leq r < 1 $, the least squares, dual least squares, Bubnov-Galerkin methods with Fourier basis will uniformly diverge to infinity at first order. The divergence effect and optimality of first order rate are confirmed in an example.
\end{abstract}
%
%
%
%
%

\section{Introduction }
Symm's integral equation of the first kind plays an important role in solving the boundary value problem of Laplace equations and constructing conformal mapping. It is generally formulated as follows. 
\newline \indent Let $ \Omega $ be a bounded domain in $ \mathrm{R}^d, \ d \geq 2 $, with boundary $ \partial \Omega $ of class $ C^2 $. Find the density $ \psi $ that solves
\begin{displaymath}
 \int_{\Gamma} \psi (y) \kappa(x-y) ds(y) = f(x), \quad x \in \Gamma,
\end{displaymath}
with the kernel
\begin{displaymath}
\kappa(x):= \left\{
\begin{array}{rcl}
 - \frac{1}{\pi} \ln \vert x \vert, & &  {d = 2},  \\
\frac{1}{2\pi} \vert x \vert^{-1},  & &  {d \geq 3}.
\end{array}
\right.
\end{displaymath}
in a strong or weak sense.
\newline \indent There is a large amount of literatures on numerical consideration of (SIE),  for example,
 \newline \indent (a) Collocation and quolocation boundary element method into two-dimensional case with the boundary $ \Gamma $ being a Lipschitz curve, see [3]; for $ \Gamma $ a closed smooth curve, see [5,17,22,24,27]; for piecewise smooth curve, see [23].
  \newline \indent (b) Wavelet-based or trigonometric-based Galerkin method into two-dimensional case with the boundary $ \Gamma $ be analytic. See [18, Chapter 3.3] and [15].
 \newline \indent (c) Galerkin boundary element method into two-dimensional case with the boundary $ \Gamma $ be a closed smooth curve, see [27]; for $ \Gamma $ be a Lipschitz curve, see [1,3]; especially for $ \Gamma $ consists of a finite number of smooth arcs of finite length, see [26].
 \newline \indent (d) Galerkin boundary element method into three-dimensional case, See [2,4,14,16], and even higher dimensions, see [13].
 \newline \indent (e) Some other methods, for example, multigrid method, see [25].
\newline \indent We are mostly interested in the numerical analysis of Petrov-Galerkin methods under Fourier basis for planar (SIE)(See [18, Chapter 3.3]). In this paper, we strengthen the result of convergence and error analysis in [18, Chapter 3.3], and develop a systematic divergence analysis, which is not found in literatures so far. We restrict $ \Omega \subset \mathrm{R}^2 $ to be some bounded, simply connected region and $ \Gamma =\partial \Omega $ be a closed curve of $ C^3 $ ($ C^k $) class, that is, the boundary $ \partial \Omega $ has a $ 2\pi $ periodic, three ($ k -th $) times continuously differentiable parametrization of the form
\begin{equation}
\boldsymbol{x}=\boldsymbol{\gamma}(s) = (a(s), b(s)), \quad s \in [0,2\pi],
\end{equation}
 Throughout this paper, we assume that the representation satisfies $ \vert \dot {\boldsymbol{\gamma}} (s) \vert > 0 $ for all $ s \in [0,2\pi] $.
 \newline \indent Now the Symm's integral equation of the first kind we address in a strong sense is formulated as, to determine the density $ \psi \in C(\partial \Omega) $ that solves
\begin{equation}
 - \frac{1}{\pi} \int_{\partial \Omega} \psi (y) \ln \vert x - y \vert ds(y) = f(x), \quad x \in \partial \Omega,
\end{equation}
where $ f \in C(\partial \Omega) $ is some given function (This problem directly arise from the Dirichlet problem in Laplace problem, see [18, Chapter 3.3] for the background). Insert the representation of curve boundary (1.1) into (1.2), then Symm's equation takes the form
\begin{equation}
 - \frac{1}{\pi} \int^{2\pi}_0 \Psi (s) \ln \vert \boldsymbol{\gamma}(t) - \boldsymbol{\gamma}(s) \vert ds = g(t), \quad x \in [0,2\pi],
\end{equation}
for the transformed density $ \Psi(s) := \psi(\boldsymbol{\gamma}(s))\vert \dot {\boldsymbol{\gamma}}(s)\vert $ and $ g(t) := f(\boldsymbol{\gamma}(t)), \ s \in [0,2\pi] $.

To make convenience for proceeding analysis, we further assume on the boundary $ \partial \Omega $ that there exists $ z_0 \in \Omega $ with $ \vert x - z_0 \vert \neq 1 $ for all $ x \in \partial \Omega $. This guarantees that $ \psi(x) \in C(\partial \Omega) $ solves (1.2) for $ f = 0 $ must be trivial (See Lemma 2.4), which ensure the injectivity of Symm's integral operator in (1.3) when considering it in background space $ H^s(0,2\pi) \ (-1 \leq s < 2) $ (Also there exists another formulation on tranfinite diameter of $ \Omega $ to ensure the injectivity, see [28]).
\newline \indent For research on Symm's integral equation of form (1.3), strenghthening $ \partial \Omega $ to be analytic, classical results in [18, Chapter 3.3] show complete convergence and error analysis in $ L^2 $ setting for three Petrov-Galerkin methods (least squares, dual least squares, Bubnov-Galerkin methods under Fourier basis) with $ \Psi \in H^r (0,2\pi), \ r \geq 0 $ (equivalent to $ g \in H^s (0,2\pi), \ s \geq 1 $).

 In this paper, weakening the boundary $ \partial \Omega $ from analytic to $ C^3 $ class, we maintain the convergence and error analysis from analytic case. Besides, we present a uniform divergence result for $ g \in H^r(0,2\pi), 0 \leq r <1 $  in $ L^2 $ setting which can be formulated as follows.
\begin{theorem} Let $ g \in L^2(0,2\pi) \setminus H^1(0,2\pi) $. Then the least squares, dual least squares,  Bubnov-Galerkin methods with trigonometric polynomial all diverge with first order rate (\textbf{optimal}) to infinity for (1.3).
\end{theorem}

\indent As to the arrangement of the rest contents. In section 2, we introduce necessary preliminaries, such as periodic Sobolev space, basic properties of Symm's integral operator. In section 3, we give settings for unified Petrov-Galerkin, least squares, dual least squares, Bubnov-Galerkin methods respectively. In section 4,5,6, we analyze convergence, divergence and rate for above three specific Petrov-Galerkin settings respectively. In section 7, we give an example to confirm the first order to be uniformly optimal.
\section{Preliminaries}
\subsection{Periodic Sobolev space $ H^r(0,2\pi) $ and estimates}
Throughout this paper, we denote the $ 2\pi - $ periodic Sobolev space as $ H^r(0,2\pi) $. The following material can be found in [18,19].
\newline \indent For $ r \geq 0 $, the periodic Sobolev space $ H^r(0,2\pi) $ of order $ r $ is defined by
\begin{equation*}
H^r(0,2\pi) := \{ \sum_{k \in \mathrm{Z}} a_k e^{ikt}: \ \sum_{k \in \mathrm{Z}} (1+k^2)^r \vert a_k \vert^2 < \infty \}
\end{equation*}
We note that $ H^0 (0,2\pi) $ coincides with $ L^2(0,2\pi) $.
\newline \indent The Sobolev space $ H^r(0,2\pi) $ is a Hilbert space with the inner product defined by
\begin{equation}
(x,y)_{H^r} := \sum_{k \in \mathrm{Z}} (1+k^2)^r a_k \bar b_k ,
\end{equation}
where
\begin{equation*}
x(t) = \sum_{k \in \mathrm{Z}}  a_k e^{ikt} \quad \textrm{and} \quad y(t) = \sum_{k \in \mathrm{Z}}  b_k e^{ikt}.
\end{equation*}
The norm in $ H^r(0,2\pi) $ is given by
\begin{equation*}
\Vert x \Vert_{H^r} = (\sum_{k \in \mathrm{Z}} (1+k^2)^r \vert a_k \vert^2)^{\frac{1}{2}}.
\end{equation*}

\indent For $ r \geq 0 $, we denote by $ H^{-r} (0,2\pi) $ the dual space of $ H^r (0,2\pi) $, i.e., the space of all linear bounded functionals on $ H^r (0,2\pi) $.
\newline \indent By Riesz representation theorem, $ H^{-r} (0,2\pi) $ can be represented in terms of Fourier coefficients. For $ F \in H^{-r} (0,2\pi) $, define $ c_k := F(e^{ikt}) $ for $ k \in \mathrm{Z} $. Then the norm is given by
\begin{equation*}
\Vert F \Vert_{H^{-r}} = (\sum_{k \in \mathrm{Z}} (1+k^2)^{-r} \vert c_k \vert^2 )^{\frac{1}{2}}.
\end{equation*}
Conversely, let $ c_m \in \mathrm{C} $ satify
\begin{equation*}
 \sum_{k \in \mathrm{Z}} (1+k^2)^{-r} \vert c_k \vert^2  < \infty.
\end{equation*}
Then there exists a bounded linear functional $ F $ on $ H^r(0,2\pi) $ with $ F(e^{ikt}) = c_k $ for all $ k \in \mathrm{Z} $.
\newline \indent For each $ g \in L^2(0,2\pi) $ the sesquilinear duality pairing
\begin{equation}
G(\varphi) := \frac{1}{2\pi} \int^{2\pi}_0 \varphi(t) \bar g(t) dt, \quad \varphi \in H^r(0,2\pi),
\end{equation}
defines a linear functional $ G \in H^{-r} (0,2\pi) $. In this sense, $ L^2(0,2\pi) $ is a subspace of each dual space $ H^{-r} (0,2\pi) $, and the trigonometric polynomials are dense in $ H^{-r} (0,2\pi) $.
\newline \indent $ H^{-r} (0,2\pi) $ becomes a Hilbert space by appropriately extending the inner product (2.1) to negative order $ -r $, i.e., let
\begin{equation*}
(F,G)_{H^{-r}} = \sum_{k \in \mathrm{Z}} (1+k^2)^{-r} c_k \bar d_k
\end{equation*}
where $ c_k := F(e^{ikt}), \ d_k := G(e^{ikt}) $. For $ r = 0 $ the duality map (2.2)  is bijective with $ \Vert G \Vert_{H^{-0}} = \Vert g \Vert_{H^0} $. Therefore, we can identify $ H^{-0} $ and $ H^0 $ and obtain a Hilbert scale of Hilbert spaces $ \{H^r\}_{r \in \mathrm{R}} $.
\newline \indent Notice that, for $ r > s $, the Sobolev space $ H^r(0,2\pi) $ is a dense subspace of $ H^s(0,2\pi) $. The inclusion operator from $ H^r(0,2\pi) $ into $ H^s(0,2\pi) $ is compact.
\begin{lemma}
Let $ P_n : L^2 ( 0, 2\pi ) \longrightarrow X_n \subset L^2(0, 2\pi ) $ be an orthogonal projection operator, where $ X_n = span \{ e^{ikt} \}^n_{k = -n} $. Then \begin{math} P_n \end{math} is given as follows£º
\begin{equation*}
(P_n x)(t) = \sum^n_{k = -n} a_k e^{ikt}, \quad x \in L^2(0,2\pi),
\end{equation*}
where
\begin{equation*}
a_k = \frac{1}{2\pi}\int^{2\pi}_0 x(s) \exp(-iks) ds, \quad k \in \mathrm{N},
  \end{equation*}
 are the Fourier coefficients of $ x $. Furthermore, the following estimate holds:
 \begin{equation*}
 \Vert x- P_n x \Vert_{H^s} \leq \frac{1}{ n^{r-s} } \Vert x \Vert_{H^r} \quad x \in  \ H^r(0,2\pi),
 \end{equation*}
 where $ r \geq s $.
\end{lemma}
\begin{proof}
See [18, Theorem A.43].
\end{proof}
\begin{lemma}
(Inverse inequality in $ H^r (0,2\pi) $):
Let $ r \geq s $. Then there exists a $ c > 0 $ such that
\begin{equation*}
\Vert \psi_n \Vert_{H^r} \leq c n^{r-s} \Vert \psi_n \Vert_{H^s}, \quad \forall \ \psi_n \in X_n
\end{equation*}
for all $ n \in \mathrm{N} $.
\end{lemma}
\begin{proof}
See [18, Theorem 3.19].
\end{proof}
\subsection{Integral operator and regularity}
\begin{lemma}
Let $ r \in \mathrm{N} $ and $ k \in C^r([0,2\pi] \times [0,2\pi]) $ be $ 2\pi- $ periodic with respect to both variables. Then the integral operator $ K $, defined by
\begin{equation*}
(Kx)(t) := \int^{2\pi}_0 k(t,s) x(s) ds, \quad t \in (0,2\pi),
\end{equation*}
can be extended to a bounded operator from $ H^p(0,2\pi) $ into $ H^r(0,2\pi) $ for every $ -r \leq p \leq r $.
\end{lemma}
\begin{proof}
See [18, Theorem A.45].
\end{proof}
\subsection{Symm's integral equation of the first kind}
Throughout this paper, we denote the integral operator in (1.3) by $ K $.
\begin{equation*}
(K \Psi)(t) :=  - \frac{1}{\pi} \int^{2\pi}_0 \Psi (s) \ln \vert \boldsymbol{\gamma}(t) - \boldsymbol{\gamma}(s) \vert ds, \quad t \in (0,2\pi).
\end{equation*}
Before the divergence analysis, we first utilize the common technique (see [18, Chapter 3.3]) in Symm's integral equation of the first kind to split above kernel into two parts:
\begin{equation}
- \frac{1}{\pi} \ln \vert \boldsymbol{\gamma}(t) - \boldsymbol{\gamma}(s) \vert = - \frac{1}{2\pi} (\ln (4 \sin^2 \frac{t-s}{2}) -1 )+ k(t,s), \quad (t \neq s)
\end{equation}
where the former is the singular part with singularities at $ t = s$ and corresponds to disc with center $ 0 $ and radius $ a = e^{-\frac{1}{2}} $, that is, $ \boldsymbol{\gamma}_a (s)= a ( \cos s, \sin s ), \ s \in [0,2\pi] $. The latter part $ k $  has a $ C^2 $ continuation onto $ [0,2\pi] \times [0,2\pi] $ (See Appendix B) since $ \boldsymbol{\gamma} $ is three times continuously differentiable.
They define two integral operators respectively as
\begin{equation}
(K_0 \Psi)(t) := - \frac{1}{2\pi} \int^{2\pi}_0 \Psi (s) [\ln (4 \sin^2 \frac{t-s}{2}) - 1 ] ds
\end{equation}
\begin{equation}
C \Psi := K \Psi - K_0 \Psi = \int^{2\pi}_0 k(t,s) \Psi(s) ds.
\end{equation}
We now recall some useful results for above operators.
\begin{lemma}
Let $ \Omega \subseteq \mathrm{R}^2 $ be a simply connected bounded domain with $ \partial \Omega $ be its boundary belongs to class of $ C^2 $. Suppose there exists $ z_0 \in \Omega $ with $ \vert x - z_0 \vert \neq 1 $ for all $ x \in \partial \Omega $. Then $ \psi(x) \in C(\partial \Omega) $ solves (1.2) for $ f = 0 $ must be $ \psi = 0 $.
\end{lemma}
\begin{proof}
See [19, Theorem 7.38]
\end{proof}
\begin{lemma}
It holds that
\begin{equation*}
\frac{1}{2\pi} \int^{2\pi}_0 e^{ins} \ln (4 \sin^2 \frac{s}{2}) ds =
\left\{
\begin{array}{rcl}
-\frac{1}{\vert n \vert}, & &  {n \in \mathrm{Z}, n \neq 0},  \\
0,  & &  {n = 0}.
\end{array}
\right.
\end{equation*}
This gives that the functions
\begin{equation*}
\hat \psi_n (t) := e^{int}, \quad t \in [0,2\pi], \ n \in \mathrm{Z},
\end{equation*}
are eigenfunctions of $ K_0 $:
\begin{equation*}
K_0 \hat \psi_n = \frac{1}{\vert n \vert} \hat \psi_n \quad \textrm{for} \ n \neq 0 \ \textrm{and}
\end{equation*}
\begin{equation*}
K_0 \hat \psi_0 = \hat \psi_0.
\end{equation*}
\end{lemma}
\begin{proof}
See [18, Theorem 3.17]
\end{proof}
\begin{lemma}
Let $ \Omega \subseteq \mathrm{R}^2 $ be a simply connected bounded domain with $ \partial \Omega $ be its boundary belongs to class of $ C^3 $. Suppose there exists $ z_0 \in \Omega $ with $ \vert x - z_0 \vert \neq 1 $ for all $ x \in \partial \Omega $. Then
 \newline (a) $ K $ is a compact and self-joint in $ L^2(0,2\pi) $.
 \newline (b) The operator $ K_0 $ is bounded injective from $ H^{s-1} (0,2\pi) $ onto $ H^s(0,2\pi) $ with bounded inverses for every $ s \in \mathrm{R} $, the same assertion also holds for $ K $ when $ -1 \leq s < 2 $.
 \newline (c) The operator $ K_0 $ is coercive from $ H^{-\frac{1}{2}} (0,2\pi) $ into $ H^{\frac{1}{2}} (0,2\pi) $.
 \newline (d) The operator $ C := K- K_0 $ is compact from $ H^{s-1} (0,2\pi) $ into $ H^s(0,2\pi) $ for every $ -1 \leq s < 2 $.
\end{lemma}
\begin{proof}
See [18, Theorem A.33 and Theorem 3.18] for (a), the former part of (b), (c). Following the main idea in [1, theorem 3.18], we prove the latter part of (b) and (d). Since the $ k $ has a $ C^2 $ continuation, by Lemma 2.3, $ C $ defines a bounded operator from $ H^p $ to $ H^2 $ with $ -2 \leq p \leq 2 $. Composing with a compact embedding $ H^2 \subset \subset H^s, \ (s < 2) $, (d) follows.
\newline \indent Similar to [18, theorem 3.18],  for the latter part of (b) it is sufficient to prove the injectivity of $ K $ from $ H^{s-1} $ to $ H^s $ with $ -1 \leq s < 2 $. Let $ \Psi \in H^{s-1} $ with $ K \Psi = 0 $. From $ K_0 \Psi = - C \Psi $ and the mapping properties of $ C $, we know $ K_0 \Psi \in H^2 (0,2\pi) $ and thus, $ \Psi \in H^1 (0,2\pi) $. This implies that $ \Psi $ is continuous and the transformed function $ \psi (\boldsymbol{\gamma}(t)) = \frac{\Psi(t)}{\vert \boldsymbol{\gamma}'(t) \vert} $ satifies (1.2) for $ g = 0 $. Lemma 2.4 gives $ \psi = 0 $.
\end{proof}
\subsection{Gelfand triple, coercivity and G\"{a}rding's inequality}
Let $ V $ be reflexive Banach space with dual space $ V^* $. We denote the norms in $ V $ and $ V^* $ by $ \Vert \cdot \Vert_V $ and $ \Vert \cdot \Vert_{V^*} $, respectively. A linear bounded operator $ A : V^* \to V $ is called coercive if there exists a $ \gamma > 0 $ such that \begin{equation*}
 \Re \langle x, Ax \rangle \geq \gamma \Vert x \Vert^2_{V^*} \quad \textrm{for all} \ x \in V^* ,
\end{equation*}
with dual pairing $ \langle \cdot, \cdot \rangle $ in $ (V^*,V) $. The operator $ A $ satisfies G\"{a}rding's inequality if there exists a linear compact operator $ C : V^* \to V $ such that $ K + C $ is coercive, that is,
\begin{equation*}
 \Re \langle x, Ax \rangle \geq \gamma \Vert x \Vert^2_{V^*} - \Re \langle x, Cx \rangle \quad \textrm{for all} \ x \in V^* ,
\end{equation*}

\indent A Gelfand triple $ (V, X, V^*) $ consists of a reflexive Banach space $ V $, a Hilbert space $ X $, and the dual space $ V^* $ of $ V $ such that
\newline \indent (a) $ V $ is dense subspace of $ X $, and
\newline\indent (b) the embedding $ J : V \to X $ is bounded.
\newline \indent We write $ V \subseteq X \subseteq V^* $ because we can identify $ X $ with a dense subspace of $ V^* $. This identification is given by the dual operator $ J^* : X \to V^* $ of $ J $, where we identify the dual of the Hilbert space $ X $ by itself and $ (x,y) = \langle J^* x, y \rangle $ for all $ x \in X $ and $ y \in V $.

\section{Unified projection setting and its divergence result}
Let $ X,Y $ be Hilbert spaces over the complex scalar field, $ \{ X_n \} $ and $ \{ Y_n \} $ be sequences of closed subspaces of $ X $ and $ Y $ respectively, $ P_n := P_{X_n} $ and $ Q_n:= Q_{Y_n} $ be orthogonal projection operators which project $ X $ and $ Y $ onto $ X_n $ and $ Y_n $ respectively. Let the original operator equation of the first kind be
\begin{equation}
Ax = b, A \in \mathcal{B}(X,Y), \ x \in X, \ b \in Y
\end{equation}
Its unified projection approximation setting is
\begin{equation}
A_n x_n = b_n, \ A_n \in \mathcal{B}(X_n,Y_n), \ x_n \in X_n, \ b_n \in Y_n,
\end{equation}
where
\begin{equation*}
A_n := Q_n A P_n : X_n \to Y_n, \  \mathcal{R}(A_n) \ \textrm{closed}.
\end{equation*}
Specifically, three different projectional setting is arranged as
\newline (1) Least squares method: Finite-dimensional $ X^{LS}_n \subseteq X $ such that $ \bigcup_{n \in \mathrm{N}} X^{LS}_n $ is dense in $ X $ with $ Y^{LS}_n = A(X^{LS}_n) $ and $ b^{LS}_n := Q^{LS}_n b $, where $ Q^{LS}_n := Q_{Y^{LS}_n}$;
\newline (2) Dual least squares method: Finite-dimensional $ Y^{DLS}_n \subseteq Y $ such that $ \bigcup_{n \in \mathrm{N}} Y^{DLS}_n $ is dense in $ Y $ with $ X^{DLS}_n = A^*(Y^{DLS}_n) $ and $ b^{DLS}_n := Q^{DLS}_n b $, where $ Q^{DLS}_n := Q_{Y^{DLS}_n}$ ;
\newline (3) Bubnov-Galerkin method: Backgound Hilbert spaces $ X = Y $ with finite-dimensional $ Y^{BG}_n = X^{BG}_n \subseteq X $ such that $ \bigcup_{n \in \mathrm{N}} X^{BG}_n $ is dense in $ X $ and $ b^{BG}_n := Q^{BG}_n b $, where $ Q^{BG}_n := Q_{Y^{BG}_n}$ .

The Unified divergence result for general projection setting is illustrated as follows.
\begin{lemma}
For projection setting (3.1), (3.2), if $ (\{ X_n \}_{n \in \mathrm{N}}, \{ Y_n \}_{n \in \mathrm{N}} ) $ satisfies the completeness condition, that is,
\begin{equation*}
P_n \stackrel{s}{\longrightarrow} I_X, \quad
 Q_n \stackrel{s}{\longrightarrow} I_Y,
\end{equation*}
and
\begin{equation}
\sup_n \Vert A^\dagger_n Q_n A \Vert < \infty
\end{equation}
where $ \dagger $ denotes the Moore-Penrose inverse of linear operator (See [2, Definition 2.2]),
then, for $ b \notin \mathcal{D}(A^\dagger) = \mathcal{R}(A) \oplus \mathcal{R}(A)^\perp $,
\begin{equation*}
\lim_{n \to \infty} \Vert A^\dagger_n Q_n Q_{\overline{\mathcal{R}(A)}} b \Vert = \infty
\end{equation*}
\end{lemma}
\begin{proof}
See [9, Theorem 2.2 (c)]
\end{proof}

\section{Analysis for Least square method}
Set $ X = Y = L^2(0,2\pi) $ and
\begin{equation}
X^{LS}_n = span \{e^{ikt} \}^{n}_{k=-n}, \quad Y^{LS}_n = K(X^{LS}_n)
\end{equation}
Aiming to complete the convergence analysis on least squares method, we first introduce some technical lemmas:
\begin{lemma}
Let $ A: X \to Y $ be a linear, bounded, and injective operator between Hilbert spaces and $ X^{LS}_n \subseteq X $ be finite-dimensional subspaces such that $ \bigcup_{n \in \mathrm{N}} X^{LS}_n  $ is dense in $ X $. Let $ x \in X $ be the solution of $ Ax = y $ and $ x^{\delta}_n $ be the least square solution from (3.2) with $ b $ being replaced by $ b^\delta $ and $ \Vert b^\delta - b \Vert \leq \delta $. Define
\begin{equation*}
\sigma^{LS}_n = \sigma^{LS}_n (A) := \max \{ \Vert z_n \Vert : z_n \in X^{LS}_n, \Vert Az_n \Vert = 1 \},
\end{equation*}
let there exists a constant $ \tau^{LS} > 0 $, independent of $ n $, such that
\begin{equation}
\min_{z_n \in X^{LS}_n} \{ \Vert x - z_n \Vert + \sigma_n \Vert A(x-z_n) \Vert \} \leq \tau^{LS} \Vert x \Vert \  \textrm{for all} \ x \in X.
\end{equation}
Then the least square method is uniquely solvable, that is, $ A^{LS}_n:= Q_{Y^{LS}_n} A P_{X^{LS}_n} : X^{LS}_n \to Y^{LS}_n $ is invertible, where $ Y^{LS}_n = A(X^{LS}_n) $, and convergent, that is,
\begin{equation}
{A^{LS}_n}^{-1} Q^{LS}_n b \stackrel{s}{\to} A^{-1} b, \quad (b \in \mathcal{R}(A))
\end{equation}
with $ \Vert R^{LS}_n \Vert \leq \sigma^{LS}_n $, where $ R^{LS}_n:= {A^{LS}_n}^{-1} Q^{LS}_n : Y \to X^{LS}_n \subseteq X $. In this case, we have the error estimate
\begin{equation*}
\Vert A^{-1} b - {A^{LS}_n}^{-1} Q^{LS}_n b^{\delta} \Vert \leq \sigma^{LS}_n \delta + c^{LS} \min\{\Vert x - z_n \Vert: z_n \in X^{LS}_n\}
\end{equation*}
where $ c^{LS}:= \tau^{LS} +1, $
   Notice that $ (\{X^{LS}_n\}, \{Y^{LS}_n\})_{n \in \mathrm{N}}$ are all not specifically chosen.
\end{lemma}
\begin{proof}
This is an operator equation version of [18, Theorem 3.10].
\end{proof}
\begin{remark}
Notice that if the conditions of Lemma 4.1 is satisfied, then the (4.3) holds. Inserting $ b = Ax, \ x \in X $ into the (4.3), we have \begin{equation*}
{A^{LS}_n}^{-1} Q^{LS}_n A x \stackrel{s}{\to} x = A^{-1} A x, \quad ( x \in X)
\end{equation*}
Then Banach-Steinhaus theorem gives that
\begin{equation*}
\sup_n \Vert {A^{LS}_n}^{-1} Q^{LS}_n A \Vert < \infty
\end{equation*}
which by Lemma 3.1 gives a divergence result if $ \overline{\mathcal{R}(A)} = Y $. In this way, we may complete both convergence and divergence analysis after verifying the conditions of Lemma 4.1 for $ K $ in $ L^2(0,2\pi) $. The same thought can also be applied to the other two Petrov-Galerkin methods.
\end{remark}
\begin{lemma}
(Stability estimate for Symm's integral equation of the first kind):
There exists a $ c>0 $, independent of $ n $, such that
\begin{equation*}
\Vert \Psi_n \Vert_{L^2} \leq c n \Vert K \Psi_n \Vert_{L^2} \ \textrm{for all}  \ \Psi_n \in X_n.
\end{equation*}
This yields that $ \sigma^{LS}_n(K) \leq c n $.
\end{lemma}
\begin{proof}
See [18, Lemma 3.19].
\end{proof}
In the following, we prove that (4.2) holds for $ K $ in $ L^2(0,2\pi) $.
\begin{proof}
Choosing $ z_n = P^{LS}_n x $, we have
\begin{equation*}
\min_{z_n \in X^{LS}_n} \{ \Vert x - z_n \Vert + \sigma_n \Vert K (x-z_n) \Vert \}
\end{equation*}
\begin{equation*}
\leq \Vert x - P^{LS}_n x  \Vert + \sigma_n(K) \Vert K ( x - P^{LS}_n x) \Vert
\end{equation*}
\begin{equation}
\leq 2 \Vert x \Vert + c n \Vert K ( x - P^{LS}_n x) \Vert  \quad \textrm{by Lemma 4.2},
\end{equation}
where $ c > 0 $ is a constant independent of $ n $. Applying Lemma 2.6 (b) with $ s = 0 $, we know that $ K $ is bounded from $ H^{-1}(0,2\pi) $ onto $ L^2(0,2\pi) $, thus,
\begin{equation*}
\Vert K ( x - P^{LS}_n x) \Vert_{L^2} \leq \Vert K \Vert_{H^{-1} \to L^2} \Vert x - P^{LS}_n x \Vert_{H^{-1}} \  (L^2(0,2\pi) \subseteq H^{-1} (0,2\pi))
\end{equation*}
\begin{equation*}
 \leq \Vert K \Vert_{H^{-1} \to L^2} \frac{1}{n} \Vert x \Vert_{L^2} \quad \textrm{for all} \ x \in L^2(0,2\pi).
\end{equation*}
with Lemma 2.1 of  $ r = 0 $ and $ s = -1 $. Together with (4.5), it yields that
\begin{equation*}
\min_{z_n \in X^{LS}_n} \{ \Vert x - z_n \Vert + \sigma_n \Vert K (x-z_n) \Vert \} \leq   (2+c \Vert K \Vert_{H^{-1} \to L^2} ) \Vert x \Vert_{L^2}.
\end{equation*}
This complete the proof of (4.2) in $ L^2(0,2\pi) $.
\end{proof}
\indent  Thus by Lemma 4.1, we have error estimate for least squares method
\begin{equation*}
\Vert K^{-1} b - {K^{LS}_{n}}^{-1} Q^{LS}_n b^{\delta} \Vert_{L^2} \leq C^{LS} n \delta + C^{LS}_1 \min\{  \Vert K^{-1} b - z_n \Vert_{L^2}: z_n \in X^{LS}_n \}
\end{equation*}
where $ C^{LS}_1:= 3 +C^{LS} \Vert K \Vert_{H^{-1} \to L^2}  $.
With further regularity assumption on exact solution $ K^{-1} b \in H^r(0,2\pi), \ (r \leq 2)$, that is, $ b \in H^{r+1} (0,2\pi) $, by Lemma 2.1,
\begin{equation*}
\Vert K^{-1} b - {K^{LS}_{n}}^{-1} Q^{LS}_n b^{\delta} \Vert_{L^2} \leq C^{LS} n \delta + \frac{C^{LS}_1}{n^r} \Vert x \Vert_{H^r}
\end{equation*}
Choosing $ n = \delta^{-\frac{1}{r+1}} $, we have
 \begin{equation*}
\Vert K^{-1} b - {K^{LS}_{n}}^{-1} Q^{LS}_n b^{\delta} \Vert_{L^2} = O(\delta^{\frac{r}{r+1}}).
  \end{equation*}
This is optimal since we can examine that the rate $ O(\delta^{\frac{2\mu}{2\mu+1}}) $ is obtained for $ K^{-1} b \in \mathcal{R}({(K^* K)}^\mu) \subseteq H^{2\mu}(0,2\pi), \quad \mu = \frac{1}{2} \ \textrm{or} \ 1 $.
\newline \indent In the following, we complete the divergence analysis by Lemma 3.1. The key point that $ (3.3) $ holds for $ K $ in $ L^2(0,2\pi) $ is transformed into (4.2) (by Lemma 4.1) which is proven in Proof 9. Besides, we verify the completeness condition for $ (\{ X^{LS}_n \}, \{ Y^{LS}_n \}) $ in (4.1) (See Appendix A). Thus we have the divergence result for least square method as
\begin{theorem}
For $ b \in L^2(0,2\pi) \setminus H^1(0,2\pi) $, the least squares method with Fourier basis for Symm's integral equation diverges.
\end{theorem}
\begin{proof}
The proof of (4.2) for $ K $ in $ L^2(0,2\pi) $ with Lemma 4.1 gives that, for every $ b \notin \mathcal{D}(K^\dagger) = \mathcal{R}(K) \oplus \mathcal{R}(K)^\perp $,
\begin{equation*}
\lim_{n \to \infty} \Vert {K^{LS}_n}^{-1} Q^{LS}_n Q_{\overline{\mathcal{R}(K)}} b \Vert_{L^2} = \infty.
\end{equation*}
Since application of Lemma 2.6 (b) with $ s = 1 $ gives $ \mathcal{R}(K) = H^1(0,2\pi) $, with the fact that $ H^1(0,2\pi) $ is dense in $ L^2(0,2\pi) $, we have $ \mathcal{R}(K)^\perp = \overline{\mathcal{R}(K)}^\perp = 0 $ and $ Q_{\overline{\mathcal{R}(K)}} = I_{L^2} $.
This yields that, for $ b \in L^2(0,2\pi) \setminus H^1(0,2\pi) $,
\begin{equation*}
\lim_{n \to \infty} \Vert {K^{LS}_n}^{-1} Q^{LS}_n b \Vert_{L^2} = \infty.
\end{equation*}
This complete the proof.
\end{proof}
\begin{remark}
Actually the proof of (4.2) in $ L^2(0,2\pi) $ should have been contained in the convergence analysis part of least square method in [18, Theorem 3.20]. However, they are all missing. Here we supplement the estimate also to support analysis there.
\end{remark}
The third item in Lemma 4.1 with Lemma 4.2 gives that $ \Vert {K^{LS}_n}^{-1} Q^{LS}_n \Vert_{L^2 \to L^2} \leq cn $. Together with Theorem 4.1, it leads to the divergence rate result.
\begin{theorem}
For $ b \in L^2(0,2\pi) \setminus H^1(0,2\pi) $, the least squares method with Fourier basis with $ \Vert K^{LS}_n Q^{LS}_n b \Vert_{L^2} = O(n). $
\end{theorem}

\section{Analysis for Dual least square method}
For dual least square method with $ X = Y = L^2(0,2\pi) $, set
\begin{equation}
  Y^{DLS}_n = span \{ e^{ikt} \}^n_{k=-n} $, \quad $ X^{DLS}_n = K^* (Y^{DLS}_n),
\end{equation}
We first prepare a convergence lemma for dual least squares method
\begin{lemma}
Let $ X $ and $ Y $ be Hilbert spaces and $ A: X \to Y $ be a linear, bounded, and injective such that the range $ \mathcal{R}(A) $ is dense in $ Y $. Let $ Y^{DLS}_n \subseteq Y $ be finite-dimensional subspaces such that $ \bigcup_{n \in \mathrm{N}} Y^{DLS}_n  $ is dense in $ Y $. Then the dual least square method is uniquely solvable, that is, $ A^{DLS}_n:= Q_{Y^{DLS}_n} A P_{X^{DLS}_n} : X^{QLS}_n \to Y^{QLS}_n $ is invertible, where $ X^{DLS}_n = A^*(Y^{DLS}_n) $, and convergent, that is,
\begin{equation}
{A^{QLS}_n}^{-1} Q^{QLS}_n b \stackrel{s}{\to} A^{-1} b, \quad (b \in \mathcal{R}(A))
\end{equation}
with $ \Vert R^{QLS}_n \Vert \leq \sigma^{QLS}_n $, where
 \begin{equation*}
 \sigma^{QLS}_n := \max \{ \Vert z_n \Vert: z_n \in Y^{DLS}_n, \Vert A^* (Y^{DLS}_n) \Vert = 1 \}
 \end{equation*}
and $ R^{QLS}_n:= {A^{QLS}_n}^{-1} Q^{QLS}_n : Y \to X^{QLS}_n \subseteq X $. Furthermore, we have error estimate
 \begin{equation*}
 \Vert A^{-1} b - {A^{DLS}_n}^{-1} Q^{DLS}_n b^\delta \Vert \leq \sigma^{DLS}_n \delta + c \min \{ \Vert A^{-1} b - z_n \Vert : z_n \in A^*(Y_n) \}
 \end{equation*}
 where $ c =2 $. Notice that $ (\{X^{QLS}_n\}, \{Y^{QLS}_n\})_{n \in \mathrm{N}}$ are all not specifically chosen.
\end{lemma}
\begin{proof}
This is an operator equation version of [18 , Theorem 3.11].
\end{proof}
\begin{remark}
If the conditions of Lemma 5.1 is satisfied, then by the same sake in Remark 4.1, we have $ \sup_n \Vert {A^{DLS}_n}^{-1} Q^{DLS}_n A \Vert < \infty $, the key point for divergence analysis of dual least squares methods holds.
\end{remark}
Application of Lemma 2.6 (b) of $ s = 1 $ guarantee the injectivity and dense range of $ K $ in $ L^2(0,2\pi) $. Since $ K $ is self-adjoint (Lemma 2.6 (a)), the third item in Lemma 5.1 with $ Y^{QLS}_n = X^{LS}_n $ gives that $ \sigma^{QLS}_n(K) \leq c^{QLS} n $. Thus, we have error estimate for dual least squares method
\begin{equation*}
\Vert K^{-1} b - {K^{QLS}_{n}}^{-1} b^{\delta} \Vert_{L^2} \leq c^{QLS} n \delta + 2 \min\{  \Vert K^{-1} b - z_n \Vert_{L^2}: z_n \in K (Y^{DLS}_n) \}
\end{equation*}
The $ K (Y^{DLS}_n) $ prevent us to construct a more precise estimate by using Lemma 2.1.
\newline \indent Now we start divergence analysis on dual least squares method. Since $ K $ is self-adjoint in $ L^2(0,2\pi) $ (Lemma 2.6 (a)), $ X^{DLS}_n = K (Y^{DLS}_n)$, similar to the least square case, we can verify that $ (\{ X^{DLS}_n \}, \{Y^{DLS}_n \}) $ satisfies the completeness condition. Thus we have the following divergence result
\begin{theorem}
For $ b \in L^2(0,2\pi) \setminus H^1(0,2\pi) $, the dual least square method with Fourier basis diverges for Symm's integral equation of the first kind, that is,
\begin{equation*}
\lim_{n \to \infty} \Vert {K^{DLS}_n}^{-1} Q^{DLS}_n b \Vert_{L^2} = \infty.
\end{equation*}
Furthermore, the assertion holds for arbitrary $ L^2(0,2\pi) $ basis $ \{ \xi_k \}^{\infty}_{k=1} $,  for instance, wavelet, piecewise constant, Legendre polynomials and so on. The same proof can be applied without change.
\end{theorem}
\begin{remark}
  Lemma  5.1 implies that, for $ b \in H^1(0,2\pi) $, the dual least square method with arbitrary $ L^2(0,2\pi) $ basis converges, that is,
\begin{equation*}
{K^{DLS}_n}^{-1} Q^{DLS}_n b \stackrel{s}{\to} K^{-1} b \quad (n \to \infty).
\end{equation*}
Here we give a complete division to all $ b \in L^2(0,2\pi) $ for convergence or divergence in dual least square method with arbitrary $ L^2(0,2\pi)$ basis.
\end{remark}
The divergence result with $ \sigma^{DLS}_n (K) = \sigma^{LS}_n (K) \leq cn $ gives
\begin{theorem}
For $ b \in L^2(0,2\pi) \setminus H^1(0,2\pi) $, the dual least square method with Fourier basis diverges with rate $ O(n) $, that is,
$ \Vert {K^{DLS}_n}^\dagger Q^{DLS}_n b \Vert_{L^2} = O(n) $
\end{theorem}
 Now we know the dual least squares method with Fourier basis diverges for $ b \in H^r, \ 0< r <1 $ in $ L^2 $ norm. This fact motivates us to find a convergence for $ b \in H^r, \ 0 \leq r <1 $ in a weaker setting.  Thus we further consider the convergence in $ H^{-1} $. Let $ K $ maps from $H^{-1}(0,2\pi) $ to $ L^2(0,2\pi) $ (Lemma 2.6 (b) $ s= 0 $). The application of Lemma 6.1 directly gives that
\begin{equation*}
\Vert  {K^{DLS}_n}^\dagger Q^{DLS}_n b - K^{-1} b \Vert_{H^{-1}} \leq 2 \min \{ \Vert K^{-1} b - z_n \Vert_{H^{-1}}: z_n \in K^*(Y^{DLS}_n) \} \to 0.
\end{equation*}
\begin{remark}
Notice that, if we see $ K $ as an operator mapping from $ H^{-1}(0,2\pi) $ to $ L^2(0,2\pi) $, then $ \mathcal{R}(K) = L^2(0,2\pi) $, $ K^{-1} $ is bounded, that is, $ Kx = b $ is well-posed. Hence we have no need to consider the influence of noise in error estimate.
\end{remark}
\section{Analysis for Bubnov-Galerkin method}
Set $ X =Y = L^2(0,2\pi) $ and $ X^{BG}_n = Y^{BG}_n = span \{ e^{ikt}\}^n_{k=-n} $, then $ (\{ X^{BG}_n \}, \{Y^{BG}_n \}) $ satisfies the completeness condition. To prepare the convergence analysis, we introduce a basic lemma:
\begin{lemma}
Let $ (V, X, V^* ) $ be a Gelfand triple, and $ X^{BG}_n \subseteq V $ be finite-dimensional subspaces such that $ \bigcup_{n \in \mathrm{N}} X^{BG}_n $ is dense in $ X $. Let $ A: V^* \to V $ be one-to-one and satisfies G\"{a}rding's inequality with some compact operator $ C: V^* \to V $, that is, there exists $ \gamma > 0 $ such that
 \begin{equation*}
 \Re \langle x, Ax \rangle \geq \gamma \Vert x \Vert^2_{V^*} -  \Re \langle x, Cx \rangle , \quad (\textrm{for all} \ x \in V^*).
 \end{equation*}
 Then
 \newline \indent (a) the Bubnov-Galerkin system is uniquely solvable, that is, $ A^{BG}_n := P^{BG}_n A P^{BG}_n : X^{BG}_n \to X^{BG}_n $  is invertible, where $ X = Y $ and $ X^{BG}_n = Y^{BG}_n $, and converge in $ V^* $ with
 \begin{equation*}
 \Vert A^{-1} b - {A^{BG}_n}^{-1} P^{BG}_n b^\delta \Vert_{V^*} \leq c^{BG}_1 \min \{ \Vert x - z_n \Vert_{V^*}: z_n \in X^{BG}_n \}
 \end{equation*}

 \indent (b) Furthermore, if there exists $ c > 0 $ with
\begin{equation}
\Vert u - P^{BG}_n u \Vert_{V^*} \leq \frac{c}{\rho_n} \Vert u \Vert \quad {\textrm{for all}} \ u \in X
\end{equation}
 then the Bubnov-Galekrin method is also convergent in $ X $, that is,
\begin{equation}
{A^{BG}_n}^{-1} P^{BG}_n b \stackrel{s}{\to} A^{-1} b, \quad (b \in \mathcal{R}(A))
\end{equation}
with
 \begin{equation*}
 \Vert A^{-1} b - {A^{BG}_n}^{-1} P^{BG}_n b^\delta \Vert \leq \frac{1}{\gamma} \rho^2_n \delta + c^{BG} \min \{ \Vert x- z_n \Vert : z_n \in X^{BG}_n \},
 \end{equation*}
where $ R^{BG}_n:= {A^{BG}_n}^{-1} P^{BG}_n : X \to X^{BG}_n \subseteq X $, and $ \Vert R^{BG}_n \Vert \leq \frac{1}{\gamma} \rho^2_n $,
 \begin{equation*}
 \rho_n := \max \{ \Vert z_n \Vert: z_n \in X^{BG}_n, \Vert z_n \Vert_{V^*} = 1 \},
 \end{equation*}
 \begin{equation*}
  c:= \tau +1, \quad \tau:= \sup_n \Vert {A^{BG}_n}^{-1} P^{BG}_n A \Vert
  \end{equation*}
 Notice that $ \rho_n $ can be seen as a local inverse embedding constant and $ (X, \{X^{QLS}_n\}_{n \in \mathrm{N}}) $ are all not specifically chosen.
\end{lemma}
\begin{proof}
 This is the operator equation version of [18, Theorem 3.15] of no noise case $ \delta = 0 $.
\end{proof}
\begin{remark}
Notice that if the conditions of Lemma 6.1 is satisfied, then $ \sup_n \Vert {A^{BG}_n}^{-1} P^{BG}_n A \Vert < \infty $, By Lemma 3.1, the divergence result also holds.
\end{remark}
Following [18, Theorem 3.20], set $ V = H^{\frac{1}{2}}(0,2\pi) $ and $ V^* = H^{-\frac{1}{2}} (0,2\pi) $, with Lemma 2.6 (c) and (d) of $ s = \frac{1}{2} $, we know $ K: H^{- \frac{1}{2}} (0,2\pi) \to H^{\frac{1}{2}} (0,2\pi) $ satisfies G\"{a}rding inequality with $ - C $ defined in (2.5). Again following [18, theorem 3.20], with application of Lemma 2.2 of $ r = 0, s = - \frac{1}{2} $, we have
\begin{equation*}
\rho_n = \max \{ \Vert \psi_n \Vert_{L^2}: \psi_n \in X_n, \Vert \psi_n \Vert_{H^{-\frac{1}{2}}} = 1 \} \leq c \sqrt{n}.
\end{equation*}
By Lemma 2.1, we have
\begin{equation*}
\Vert u - P^{BG}_n u \Vert_{H^{-\frac{1}{2}}} \leq c \sqrt{n} \Vert u \Vert_{L^2} \quad {\textrm{for all}} \ u \in L^2(0,2\pi),
\end{equation*}
that is, (6.1) holds for Bubnov-Galerkin method. Now, by Lemma 6.1, we have
\begin{equation}
 \Vert K^{-1} b - {K^{BG}_{n}}^{-1} P^{BG}_n b^\delta \Vert_{L^2} \leq  c n \delta + c \Vert (I- P^{BG}_n ) K^{-1} b \Vert_{L^2}.
\end{equation}
If we further assume $ K^{-1} b \in H^r(0,2\pi), r \leq 2 $, then, by Lemma 2.1, we have
\begin{equation*}
 \Vert K^{-1} b - {K^{BG}_{n}}^{-1} P^{BG}_n b^\delta \Vert_{L^2} \leq  c n \delta + c \frac{1}{n^r} \Vert K^{-1} b \Vert_{H^r}.
\end{equation*}
As pointed in least squares case, choosing $ n = \delta^{-\frac{1}{r+1}} $, we have optimal convergence rate for Bubnov-Galerkin method
 \begin{equation*}
\Vert K^{-1} b - {K^{BG}_{n}}^{-1} Q^{BG}_n b^{\delta} \Vert_{L^2} = O(\delta^{\frac{r}{r+1}}).
  \end{equation*}
\begin{theorem}
For $ b \in L^2(0,2\pi) \setminus H^1(0,2\pi) $, the Bubnov-Galerkin method with Fourier basis for Symm's integral equation diverges, that is,
\begin{equation*}
\lim_{n \to \infty} \Vert {K^{BG}_n}^{-1} P^{BG}_n b \Vert_{L^2} = \infty.
\end{equation*}
\end{theorem}
The fact
\begin{equation*}
\rho_n = \max \{ \Vert \psi_n \Vert_{L^2}: \psi_n \in X_n, \Vert \psi_n \Vert_{H^{-\frac{1}{2}}} = 1 \} \leq c \sqrt{n}.
\end{equation*}
with the third item in Lemma 6.1 (b) gives that
\begin{theorem}
For $ b \in L^2(0,2\pi) \setminus H^1(0,2\pi) $, the Bubnov-Galerkin method with trigonometric basis diverges with rate $ O(n) $, that is,
$ \Vert {K^{BG}_n}^{-1} P^{BG}_n b \Vert_{L^2} = O(n) $.
\end{theorem}
\begin{remark}
The core of divergence rate all locates in estimate of $ \Vert {K^\tau_n}^\dagger Q^\tau_n \Vert, \ \tau = LS, DLS, BG $, which are all potentially indicated in [18, theorem 3.20].
\end{remark}
To supplement a weaker convergence result for $ b \in H^r, 0 \leq r<1 $, we consider $ K : H^{-\frac{1}{2}} (0,2\pi)\to  H^{\frac{1}{2}} (0,2\pi) $ again. Using Lemma 7.1 (a), we know
\begin{equation*}
\Vert K^{-1} b - {K^{BG}_n}^\dagger P^{BG}_n b \Vert_{H^{-\frac{1}{2}}} \leq c^{BG}_1 \min \{\Vert K^{-1} b - z_n \Vert_{H^{-\frac{1}{2}}} : z_n \in X^{BG}_n\}
\end{equation*}
Notice, with a-priori information on smoothness of $ K^{-1} b $, above estimate can be strengthened into a more precise form with application of Lemma 2.1. This provides a convergence result in $ H^{-\frac{1}{2}} $ setting for Bubnov-Galerkin method when $ b \in H^r(0,2\pi), \frac{1}{2} \leq r < 1 $.
\section{An example}
Here we give a example to verify the divergence result for the three projection methods and further confirm the first order rate to be optimal. Let us consider Symm's integral equation with $ \Omega $ is the disc with center at origin and radius $ a = e^{-\frac{1}{2}} $, that is,
\begin{equation*}
\boldsymbol{\gamma_a} (s) = e^{-\frac{1}{2}} (\cos s, \sin s),
\end{equation*}
\begin{equation*}
\ln \vert \boldsymbol{\gamma}_a(t) - \boldsymbol{\gamma}_a (s) \vert = \frac{1}{2} (\ln (4\sin^2 \frac{t-s}{2} ) - 1 ).
\end{equation*}
Now $ K = K_0 $, according to the Lemma 2.5, 2.6 (a), we have
\begin{equation*}
 Y^{LS}_n = K_0(X^{LS}_n) = X^{LS}_n = Y^{DLS}
 \end{equation*}
 \begin{equation*}
  = K^*_0 (Y^{DLS}_n) = K_0 (Y^{DLS}_n) = X^{DLS}_n = X^{BG}_n = Y^{BG}_n
\end{equation*}
This implies that the three projection method coincides. Thus, we only need to test Bubnov-Galerkin method.
\newline \indent Set
 \begin{equation*}
 b(t) = 1 + \sum_{0 \neq k \in \mathrm{Z}} \frac{1}{\vert k \vert^{\frac{1}{2} + \alpha }} e^{ikt}  \in L^2(0,2\pi) \setminus H^1(0,2\pi). \quad (\alpha \in (0,\frac{1}{2}))
 \end{equation*}
we can deduce that
\begin{equation*}
\Psi^\dagger_n = {K^{BG}_{0,n}}^{-1} P^{BG}_n b = 1 + \sum^n_{k = 1 } \vert k \vert^{\frac{1}{2} - \alpha } e^{ikt} + \sum^{-1}_{k = -n } \vert k \vert^{\frac{1}{2} - \alpha } e^{ikt}
\end{equation*}
and thus,
\begin{equation*}
c_1 n^{2 - 2\alpha} \leq \Vert \Psi^\dagger_n \Vert^2_{L^2} \leq c_2 n^2  \quad (\alpha \in (0,\frac{1}{2})).
\end{equation*}
This result verifies the divergence result and further confirm the first order divergence rate to be optimal by letting $ \alpha \to 0^+ $.

\section{Conclusion}
Our main contribution of this paper is as follows:
weakening the boundary $ \partial \Omega $ from analytic to $ C^3 $ class, we maintain the convergence and error analysis from analytic case. Besides, we transform all classical convergence criteria in Petrov-Galerkin methods into the support of divergence analysis, and thus, determine the least squares, dual least squares, Bubnov-Galerkin methods with Fourier basis all diverge to infinity at first order for $ b \in L^2(0,2\pi) \setminus H^1(0,2\pi) $. The divergence effect and optimality of first order rate are confirmed in Example 8.1.
As a supplementary result for divergence result in $ L^2 $ setting when $ b \in H^r(0,2\pi), \ 0 \leq r < 1 $, we provide convergence in $ H^{-1} $ setting for dual least squares method when $ b \in H^r(0,2\pi), \ 0 \leq r < 1 $, and convergence in $ H^{-\frac{1}{2}} $ setting for Bubnov-Galerkin method when $ b \in H^r(0,2\pi), \ \frac{1}{2} \leq r < 1 $.
\section*{Appendix A}
\begin{lemma}
 The $ (\{ X^{LS}_n \}_{n \in \mathrm{N}}, \{ Y^{LS}_n \}_{n \in \mathrm{N}} ) $ defined in (4.1) satisfies the completeness condition, that is,
\begin{equation*}
 P^{LS}_n \stackrel{s}{\longrightarrow} I_{L^2},\quad Q^{LS}_n \stackrel{s}{\longrightarrow} I_{L^2},
\end{equation*}
\end{lemma}
\begin{proof}
Since
\begin{equation*}
H^1(0,2\pi) = K ( L^2(0,2\pi) ) \quad (\textrm{by Lemma 2.6 (b) with} \ s = 1)
\end{equation*}
\begin{equation*}
= K (\overline{\bigcup_{n \in \mathrm{N}}  X^{LS}_n }) \subseteq \overline{\bigcup_{n \in \mathrm{N}} K( X^{LS}_n) }, \quad (\textrm{by Lemma 2.6 (a)}),
\end{equation*}
and thus,
\begin{equation*}
L^2(0,2\pi) = \overline{H^1(0,2\pi) } \subseteq \overline{\bigcup_{n \in \mathrm{N}} K( X^{LS}_n) } \subseteq L^2(0,2\pi),
\end{equation*}
we have
\begin{equation*}
X^{LS}_n \subseteq X^{LS}_{n+1}, \quad \overline{\bigcup_{n \in \mathrm{N}}  X^{LS}_n } = L^2(0,2\pi)
 \end{equation*}
 and
 \begin{equation*}
 Y^{LS}_n \subseteq Y^{LS}_{n+1}, \quad  \overline{\bigcup_{n \in \mathrm{N}}  Y^{LS}_n } = L^2(0,2\pi).
\end{equation*}
 Then for arbitrary $ x = y \in L^2(0,2\pi) $, there exist $ x^{LS}_n \in X^{LS}_n $ and $ y^{LS}_n \in Y^{LS}_n $ such that $ x^{LS}_n, y^{LS}_n \stackrel{s}{\to} x $. Define orthogonal projection operators $ P^{LS}_n := P_{X^{LS}_n}, Q^{LS}_n:= Q_{Y^{LS}_n } $. Since $ \Vert x - P^{LS}_n x \Vert_{L^2} \leq \Vert x - x^{LS}_n \Vert_{L^2} \to 0 $, it follows that $  P^{LS}_n \stackrel{s}{\longrightarrow} I_{L^2} $. The same deduction also gives $ Q^{LS}_n \stackrel{s}{\longrightarrow} I_{L^2} $
\end{proof}
\section*{Appendix B}
\begin{lemma}
Let $ \boldsymbol{\gamma} = \boldsymbol{\gamma}(s) = (a(s), b(s)) $ be three times continuously differentiable, then $ k = k(t,s) $ defined in (2.1) can be extended to $ C^2 ([0,2\pi] \times [0,2\pi]) $, that is, $ 2\pi - $ periodic, two times  continuously differentiable with respect to both variables. In particular,
\begin{equation*}
\lim_{s \to t } k (t,s) = - \frac{1}{\pi} (\ln \vert \boldsymbol{\gamma}' (t) \vert + \frac{1}{2}), \quad \lim_{s \to t } \frac{\partial}{\partial t } k (t,s) = - \frac{1}{2\pi} \frac{ \boldsymbol{\gamma}' (t) \cdot {\boldsymbol{\gamma}}'' (t)} { \vert \boldsymbol{\gamma}' (t) \vert^2 },
\end{equation*}
\begin{equation*}
\lim_{s \to t } \frac{\partial^2 }{\partial t^2 } k (t,s)  = - \frac{1}{\pi} \frac{1}{\vert \boldsymbol{\gamma}' (t) \vert^4 } \times
\end{equation*}
\begin{equation*}
[ \frac{1}{12} \vert \boldsymbol{\gamma}' (t) \vert^4 + \vert \boldsymbol{\gamma}' (t) \vert^2 ( \frac{1}{3} {\boldsymbol{\gamma}}''' (t) \cdot \boldsymbol{\gamma}' (t) + \frac{1}{4} \vert \boldsymbol{\gamma}'' (t) \vert^2  ) + \frac{1}{2} ( \boldsymbol{\gamma}'(t) \cdot \boldsymbol{\gamma}''(t)  )^2 ].
\end{equation*}
\end{lemma}
\begin{proof}
Recall the definition,
\begin{equation*}
k(t,s) = - \frac{1}{\pi} (\ln \frac{\vert \gamma(t) -\gamma(s) \vert}{\vert 2 \sin (\frac{t-s}{2})\vert} + \frac{1}{2}), \quad (t \neq s )
\end{equation*}
For above three limits, we only prove the single directional limit for $ s \to t^+ $. The first limit follows directly from mean-value theorem. Assume that $  s > t, s \to t^+ $, set
\begin{equation*}
\bar k(t,s) =  \ln \frac{\vert \gamma(t) -\gamma(s) \vert}{ 2 \sin \frac{s-t}{2} }, \quad (t \neq s )
\end{equation*}
then, for $ t \neq s $,
\begin{equation}
\frac{\partial}{\partial t} \bar k(t,s) =
\frac{(a(t) - a(s)) a'(t) + (b(t) - b(s))b'(t) }{\vert \gamma(t) - \gamma(s) \vert^2} + \frac{1}{2} \cot \frac{s-t}{2}.
\end{equation}
Using Taylor expansion with Peano residual, we have
\begin{equation*}
(a(t) - a(s)) a'(t) + (b(t) - b(s))b'(t) = - \vert \boldsymbol{\gamma}' (t) \vert^2 (s-t)
\end{equation*}
\begin{equation}
- \frac{1}{2} \boldsymbol{\gamma}'(t) \cdot \boldsymbol{\gamma}''(t) (s-t)^2 + o((s-t)^2),
\end{equation}
\begin{equation*}
\frac{1}{2} \cot \frac{s-t}{2} = \frac{1}{s-t} \frac{1 - \frac{1}{8}(s-t)^2 }{1- \frac{1}{24} (s-t)^2 + o((s-t)^2) }
\end{equation*}
\begin{equation}
= \frac{1}{\boldsymbol{\gamma}(t) - \boldsymbol{\gamma}(s) \vert^2} \times \frac{1}{s-t} \frac{1 - \frac{1}{8}(s-t)^2 }{1- \frac{1}{24} (s-t)^2 + o((s-t)^2) } \vert \boldsymbol{\gamma}(t) - \boldsymbol{\gamma}(s) \vert^2,
\end{equation}
and
\begin{equation}
\vert \boldsymbol{\gamma}(t) - \boldsymbol{\gamma}(s) \vert^2 = \vert \boldsymbol{\gamma}' (t) \vert^2 (s-t)^2
+ \boldsymbol{\gamma}'(t) \cdot \boldsymbol{\gamma}''(t) (s-t)^3 + o((s-t)^3).
\end{equation}
Thus, using (8.2)-(8.4), we transform (8.1) into
\begin{equation*}
\frac{\partial}{\partial t} \bar k(t,s) =
\frac{  \frac{1}{2} \frac{1 - \frac{1}{8}(s-t)^2}{1 - \frac{1}{24}(s-t)^2 + o(s-t)^2 } \boldsymbol{\gamma}'(t) \cdot \boldsymbol{\gamma}''(t) (s-t)^2 + o((s-t)^2) }{\vert \boldsymbol{\gamma}'(t) \vert^2 (s-t)^2 + o((s-t)^2) }
\quad (t \neq s ).
\end{equation*}
Then the second limit follows.
\newline \indent Differentiating both sides in (8.1) gives
\begin{equation}
\frac{\partial^2}{\partial t^2} \bar k(t,s) =
\frac{\partial}{\partial t} \bar L(t,s) + \frac{\partial}{\partial t} \frac{1}{2} \cot \frac{s-t}{2},
\end{equation}
where
\begin{equation*}
\bar L(t,s) =
\frac{(a(t) - a(s)) a'(t) + (b(t) - b(s))b'(t) }{\vert \gamma(t) - \gamma(s) \vert^2}.
\end{equation*}
The computations yields that
\begin{equation*}
\frac{\partial}{\partial t} \bar L(t,s) = \frac{1}{\vert \boldsymbol{\gamma} (t) - \boldsymbol{\gamma} (s)\vert^4} \times  \{ - \vert \boldsymbol{\gamma}' (t) \vert^4 (s-t)^2 -
\end{equation*}
\begin{equation*}
 2 \vert \boldsymbol{\gamma}'(t) \vert^2 (\boldsymbol{\gamma}'(t) \cdot \boldsymbol{\gamma}'' (t)) (s-t)^3 + [ \vert \boldsymbol{\gamma}'(t) \vert^2 ( -\frac{1}{3} \boldsymbol{\gamma}' (t) \cdot \boldsymbol{\gamma}'''(t) - \frac{1}{4} \vert \boldsymbol{\gamma}''(t) \vert^2 ) -
\end{equation*}
\begin{equation}
 \frac{3}{2} (\boldsymbol{\gamma}''(t) \cdot \boldsymbol{\gamma}' (t))^2 ] (s-t)^4
 + o((s-t)^4)  \}
\end{equation}
and
\begin{equation*}
\frac{\partial}{\partial t} \frac{1}{2} \cot \frac{s-t}{2} = \frac{1}{(s-t)^2} \cdot \frac{1}{1 - \frac{(s-t)^2}{12} + o((s-t)^2)}
\end{equation*}
\begin{equation}
= \frac{1}{\vert \boldsymbol{\gamma} (t) - \boldsymbol{\gamma} (s)\vert^4} \times  \frac{\vert \boldsymbol{\gamma} (t) - \boldsymbol{\gamma} (s)\vert^4}{(s-t)^2} \cdot \frac{1}{1 - \frac{(s-t)^2}{12} + o((s-t)^2)},
\end{equation}
where
\begin{equation*}
\vert \boldsymbol{\gamma} (t) - \boldsymbol{\gamma} (s)\vert^4 = \{ \vert \boldsymbol{\gamma}' (t) \vert^4 (s-t)^4 +
\end{equation*}
\begin{equation*}
 2 \vert \boldsymbol{\gamma}'(t) \vert^2 (\boldsymbol{\gamma}'(t) \cdot \boldsymbol{\gamma}'' (t)) (s-t)^5 + [ \vert \boldsymbol{\gamma}'(t) \vert^2 ( \frac{2}{3} \boldsymbol{\gamma}' (t) \cdot \boldsymbol{\gamma}'''(t) + \frac{1}{2} \vert \boldsymbol{\gamma}''(t) \vert^2 ) +
\end{equation*}
\begin{equation}
 (\boldsymbol{\gamma}''(t) \cdot \boldsymbol{\gamma}' (t))^2 ] (s-t)^6.
 + o((s-t)^6)  \}
\end{equation}
Using (8.5)-(8.8) gives that
\begin{equation*}
\frac{\partial^2}{\partial t^2} \bar k(t,s)
= \frac{1}{\vert \boldsymbol{\gamma} (t) - \boldsymbol{\gamma} (s)\vert^4} \times \{
\frac{ \frac{(s-t)^2}{12} + o((s-t)^2)}{1 - \frac{(s-t)^2}{12} + o((s-t)^2)} \vert \boldsymbol{\gamma}'(t) \vert^4 (s-t)^4 +
\end{equation*}
\begin{equation*}
 \vert \boldsymbol{\gamma}'(t) \vert^2 [ ( \frac{2}{3} \frac{1}{1 - \frac{(s-t)^2}{12} + o((s-t)^2)} - \frac{1}{3} ) \boldsymbol{\gamma}' (t) \cdot \boldsymbol{\gamma}'''(t) +
 \end{equation*}
\begin{equation*}
 ( \frac{1}{2}  \frac{1}{1 - \frac{(s-t)^2}{12} + o((s-t)^2)} -\frac{1}{4}) \vert \boldsymbol{\gamma}''(t) \vert^2 ]  (s-t)^4
 +
 \end{equation*}
\begin{equation*}
 (  \frac{1}{1 - \frac{(s-t)^2}{12} + o((s-t)^2)} - \frac{3}{2})  (\boldsymbol{\gamma}''(t) \cdot \boldsymbol{\gamma}' (t))^2 )
 (s-t)^4 + o((s-t)^4) \}.
\end{equation*}
Thus, the third limit follows.
\end{proof}

\section*{References}
\bibliographystyle{elsarticle-num-names.bst}

\end{document}